\documentclass[12pt,a4paper]{amsart}
\usepackage{amsfonts}

\usepackage[active]{srcltx}
\usepackage[all]{xy}
\usepackage{enumerate}
\usepackage[notref,notcite]{showkeys}
\usepackage{amsmath,amssymb,xspace,amsthm}

\newtheorem{theorem}{Theorem}[section]

\newtheorem{lemma}[theorem]{Lemma}

\newtheorem{corollary}[theorem]{Corollary}

\numberwithin{equation}{section} \theoremstyle{definition}
\newcommand\darrowa{\longrightarrow {\mkern -27mu} {\raise 6pt \hbox{ $\pi_0$}} {\mkern 16mu}}

\def\Der{\operatorname{Der}}
\def\span{\operatorname{span}}

\newcommand{\C}{\ensuremath{\mathbb C}\xspace}

\renewcommand{\a}{\ensuremath{\alpha}}
\renewcommand{\l}{\ensuremath{\lambda}}

\newcommand{\Z}{\ensuremath{\mathbb{Z}}\xspace}

\newcommand{\W}{\ensuremath{\mathcal{W}}\xspace}
\newcommand{\K}{\ensuremath{\mathcal{K}}\xspace}

\renewcommand{\phi}{\varphi}

\renewcommand{\leq}{\leqslant}
\renewcommand{\geq}{\geqslant}

\def\mh{\mathfrak{h}}

\def\sl{\mathfrak{sl}}
\def\gl{\mathfrak{gl}}

\def\l{\lambda}

\def\span{\text{span}}

\def\Der{\text{Der}}

\newcommand{\ptl}{\partial}

\newcommand{\be}{\beta}

\begin{document}
\title[Irreducible Witt modules]{Irreducible  Witt modules from Weyl modules and $\gl_{n}$-modules}
\author{Genqiang Liu, Rencai Lu, Kaiming zhao}
\date{}\maketitle


\begin{abstract}
For an irreducible module $P$ over the Weyl algebra $\K_n^+$ (resp. $\K_n$) and  an  irreducible module $M$ over the general liner Lie algebra $\gl_n$, using Shen's monomorphism, we make $P\otimes M$ into a module over the Witt algebra $W_n^+$  (resp. over $W_n$). We obtain the necessary and sufficient conditions for  $P\otimes M$
 to be  an irreducible module over
$W_n^+$ (resp. $W_n$),  and determine all submodules of $P\otimes M$ when
it is reducible. Thus we have constructed  a large family of irreducible weight modules  with many
different weight supports and many irreducible non-weight modules over $W_n^+$ and  $W_n$.
\end{abstract}
\vskip 10pt \noindent {\em Keywords:}   irreducible  module, $\gl_{n}$, Shen's monomorphism,  weight module, Weyl algebra

\vskip 5pt
\noindent
{\em 2010  Math. Subj. Class.:}
17B10, 17B20, 17B65, 17B66, 17B68

\vskip 10pt

\section{Introduction}

We denote by $\mathbb{Z}$, $\mathbb{Z}_+$, $\mathbb{N}$ and
$\mathbb{C}$ the sets of  all integers, nonnegative integers,
positive integers and complex numbers, respectively.

Let $n>1$ be an integer,  ${A_n^+ } =\C[t_1,t_2,\cdots,t_n]$ be the polynomial algebra, and  ${A_n } =\C[t_1^{\pm 1},t_2^{\pm 1},\cdots,t_n^{\pm 1}]$ be the Laurent polynomial algebra. The derivation
Lie algebra $W_n^+=\Der(A_n^+)$ is the Witt Lie algebra or Cartan type Lie algebra of vector fields with polynomial coefficients, while $W_n=\Der(A_n)$ is the Witt Lie algebra or Cartan type Lie algebra of vector fields with Laurent polynomial coefficients. These Lie algebras are colsely related to Jacobian conjecture, see \cite{Z}, \cite{HM}.
The Cartan type Lie algebras  of vector fields with formal power series coefficients are the Lie
algebras of infinite Lie groups which arose in the work of Sophus Lie around
1870 and were further studied by Elie Cartan in 1904--1908.  The general theory of representations of the
Cartan type infinite-dimensional Lie algebras was initiated only in 1973 when
A.N. Rudakov began the study of topological irreducible representations of these
Lie algebras, see \cite{R1,R2}. Rudakov's main result, roughly speaking, is that all
irreducible representations which satisfy a natural continuity condition can be
described explicitly as  quotients of induced modules. In 1999,
I. Penkov and V. Serganova \cite{PS} described  the supports of all irreducible weight modules over $W_n^+$. However, there are very few  concrete examples of irreducible
weight modules over $W_n^+$ for each  possible weight support. There are qute a lot of studies on representrations of $W_n$, see \cite{BF}, \cite{BMZ}, \cite{E1}, \cite{E2}, \cite{GLZ}, \cite{GZ}, [L1-5], \cite{MZ}.  Recently,  Billig and Futorny \cite{BF} classified all   irreducible Harish-Chandra  $W_d$-module,  Cavaness and Grantcharov \cite{CG} determined all irreducible weight $W_2^+$-modules with bounded weight multiplicities. Our goal of the present paper is to construct  irreducible weight modules  with many different weight sets and construct many irreducible non-weight modules over $W_n^+$ and $W_n$.

In 1986, Shen \cite{Sh} constructed a Lie algebra monomorphism from $W_n^+$ (resp. $W_n$) to the semidirect product Lie algebra $W_n^+\ltimes \gl_n({A}_n^+ )$ (resp. $W_n\ltimes \gl_n({A}_n )$) which  are later called
the full toroidal Lie algebras.  For  an irreducible module $P$ over the Weyl algebra $\K_n^+$ (resp. $\K_n$) and an  irreducible module $M$   over the general linear Lie algebra $\gl_n$,  using Shen's monomorphism,
the tensor product
$F(P, M)=P\otimes_{\C} M$
becomes a $W_n^+$-module (resp. $W_n$-module), see (\ref{Action1}). Let $\C^n$ be the natural $n$-dimensional representation of $\gl_n$ and let $V(\delta_k,k)$ be its
$k$-th exterior power, $k = 0,\cdots,n$. Note that $V(\delta_k,k)$ is an irreducible $\gl_n$-module for all
$k = 0,\cdots,n$. We show that  $F(P, M)$ is an irreducible module over
$W_n^+$ if and only if  $M\not\cong V(\delta_k, k)$ for any $k\in \{0, 1,\cdots, n\}$, see Theorem \ref{thm-2.1}. The modules $F(P, V(\delta_k,k))$
are generalization of  the modules of differential
$k$-forms. These modules form the following complex
$$\xymatrix
{0\ar[r] &F(P,V(\delta_0,0))\ar[r]^{\pi_0} &F(P,V(\delta_1,1))\ar[r]^{\hskip 1cm\pi_1} &\cdots \\
\cdots\ar[r]^{\pi_{n-2}\hskip 2cm}&{F(P,V(\delta_{n-1},n-1))}\ar[r]^{\hskip .8cm\pi_{n-1}} &F(P,V(\delta_n,n))\ar[r]& 0.}
$$

The differential $\pi_k$ of the complex is a $W_n^+$-module  (resp. $W_n$-module) homomorphism.
 Thus the kernel of  $\pi_k$ is a submodule of  $F(P, V(\delta_k,k))$. As a result, for $1 \leq  k \leq  n-1$, we see that  $F(P, V(\delta_k,k))$ are
reducible $W_n^+$-modules  (resp. $W_n$-module). Moreover, $F(P, V(\delta_0,0))$
 is irreducible if and only if $P$ is not isomorphic to the natural module $A_n^+ $  (resp. $A_n$).
The module $F(P,V(\delta_n,n))$ is irreducible if and only if $\sum_{k=1}^n \partial_k P=P$, see theorem \ref{3.5}.

The paper is organized as follows. In Section 2, we recall the definition of the Witt algebras $W_n^+$ and $W_n$, and the Lie algebra homomorphism from
$W_n^+$ (resp. $W_n$) to the full toroidal Lie algebra, which was first constructed by Shen in 1986, see \cite {Sh}. In Section 3, we prove the necessary and sufficient conditions for the
$W_n^+$-modules (resp. $W_n$-module) $F(P,M)$ to be simple, and determine all their submodules when they are reducible. In Section 4 we give a couple of examples using our established results.
\section{Preliminaries}

We fix the vector space $\mathbb{C}^n$ of $n\times 1$ complex  matrices.
Denote the standard basis by $\{\epsilon_1,\epsilon_2,...,\epsilon_n\}$.

\subsection{Witt algebras $W_n^+$ and $W_n$}
Let $\ptl_{i}=\frac{\ptl}{\ptl{t_i}}$ for any $i=1, 2, \cdots, n$.
Recall that  the classical Witt algebra $ W_n=\sum_{i=1}^n{ A}_n\ptl_{i}$ has the following Lie bracket:
$$[\sum\limits_{i=1}^{n}f_{i}\ptl_{i},\sum\limits_{j=1}^{n}g_j\ptl_{j}]=
\sum\limits_{i,j=1}^{n}(f_{j}\ptl_{j}
(g_{i})-g_{j}\ptl_{j}(f_{i}))\ptl_{i},$$
where all $f_{i}, g_j\in A_n$, and $ W_n^+=\sum_{i=1}^n{ A}_n^+\ptl_{i}$ is a subalgebra of $W_n$. It is well-known that both $W_n$ and $W_n^+$ are simple Lie algebras.

\subsection{ Finite-dimensional $\mathfrak{gl}_n$-modules}
We denote by $E_{i,j}$ the $n\times n$ square matrix
with $1$ as its $(i,j)$-entry and 0 as  other entries. We know that the general linear  Lie
algebra
$$\gl_n=\sum_{1\leq i, j\leq
n}\C E_{i,j}.$$
 Let $\mathfrak{H}=\span\{E_{ii}\,|\,1\le i\le n\}$ and
$\mh=\span\{h_{i}\,|\,1\le i\le n-1\}$ where
$h_i=E_{ii}-E_{i+1,i+1}$. 

Let $\Lambda^+=\{\l\in\mh^*\,|\,\l(h_i)\in\Z_+ \text{ for } i=1,2,...,n-1\}$ be the set of dominant weights with respect to $\mh$. For any
$\psi\in \Lambda^+$,  let $V(\psi)$ be  the simple $\sl_n$-module with
highest weight $\psi$. We make $V(\psi)$ into a $\gl_n$ module by
defining the action of the identity matrix $I$ as some scalar
$b\in\mathbb{C}$. We denote the resulting $\gl_n$-module as $V(\psi,b)$.

Define the fundamental weights $\delta_i\in\mh^*$ by
$\delta_i(h_j)=\delta_{i,j}$ for all $i,j=1,2,..., n-1$. It is
well-known that the module $V(\delta_1, 1)$ can be realized as the
natural representation of $\gl_n$ on $\mathbb{C}^n$ (the matrix
product), which we can write as $E_{ji}e_l=\delta_{li}e_j$.  In
particular,
\begin{equation}(ru^T)v=(u|v)r,\,\,\forall\,\, u,v,r\in \mathbb{C}^n,\end{equation}
where $u^T$ is the transpose of $u$, and $(u|v)=u^Tv$ is the standard bilinear form on $\mathbb{C}^n$. The exterior
product $\bigwedge^k(\mathbb{C}^n)=\mathbb{C}^n\wedge\cdots\wedge
\mathbb{C}^n\ \ (k\ \mbox{times})$ is a $\mathfrak{gl}_n$-module
with the action given by $$X(v_1\wedge\cdots\wedge
v_k)=\sum\limits_{i=1}^k v_1\wedge\cdots v_{i-1}\wedge
Xv_i\cdots\wedge v_k, \,\,\forall \,\, X\in \gl_n,$$ and the following
$\gl_n$-module isomorphism is well known:
\begin{equation}\label{dk}{\bigwedge}^k(\mathbb{C}^n)\cong V(\delta_k,
k),\,\forall\,\, 1\leq k\leq n-1.\end{equation}
For convenience of late use, we let $V(\delta_0,
0)$ be the 1-dimensional trivial $\gl_n$-module. We set  $\bigwedge^0(\mathbb{C}^n)=\C$ and
$v\wedge a=av$ for any $v\in\C^n, a\in\C$.

\subsection{ Full toroidal  algebras and Shen's monomorphism} Now we have  the general linear Lie
algebras
$$\gl_n({A}_n^+ )=\sum_{i,j=1}^n{A}_n^+ E_{i,j},\quad \gl_n({A}_n )=\sum_{i,j=1}^n{A}_n E_{i,j}$$ over the commutative associative algebras ${A}_n^+ $ and ${A}_n$ respectively,
and the full toroidal Lie algebras (see \cite{EJ})
$$\widehat{  W}_n^+=  W_n^+\ltimes \gl_n({A}_n^+ ),\quad \widehat{  W}_n=  W_n\ltimes \gl_n({A}_n )$$
 with Lie brackets
$$[d_1+f_1A_1, d_2+f_2A_2]=[d_1,d_2] +d_1(f_2)A_2-d_2(f_1)A_1+f_1f_2[A_1,A_2]$$
for $d_1,d_2\in  W_n$, $A_1,A_2\in \gl_n,$  and  $ f_1, f_2\in {A}_n $.

The following Lie algebra monomorphism was given by Shen in 1986, see \cite{Sh}.
\begin{lemma}[Shen's monomorphism]\label{lemma2.1}  The  linear map $\tau:  W_n\to \widehat{
 W}_n$ given by
$$\tau(\sum\limits_{i=1}^{n}f_{i}\ptl_{i})=
\sum\limits_{i=1}^{n}f_{i}\ptl_{i}+\sum_{i,j=1}^n\ptl_{i}(f_j)
E_{i,j}, \forall f_i\in{A}_n $$ is
 a Lie algebra monomorphism.\end{lemma}

Note that the restriction of the above homomorphism $\tau:  W_n^+\to \widehat{
 W}_n^+$  is also a Lie algebra homomorphism.

We have the following useful formula for any $\alpha=(\a_1,\dots,\a_n)^T\in \Z^n$:
\begin{equation}\tau(t^{\a}\partial_{j})=t^{\a}\partial_{j} + \sum_{i=1}^n\ptl_{i}(t^{\a})  E_{ij},\end{equation}
where $t^\a=t_1^{\a_1}\cdots t_n^{\a_n}$.

\subsection{ The Weyl  algebras $\K_n^+, \K_n$ and   Witt modules}

The Weyl  algebra $\mathcal{K}_n^+$ is the simple associative algebra $\C[t_1,\cdots,t_n,\partial_{1},\cdots,\partial_{n}]$, while  $\mathcal{K}_n$ is the simple associative algebra $\C[t_1^{\pm1},\cdots,t_n^{\pm1},\partial_{1},\cdots,\partial_{n}]$.


Let $P$ be a module over the associative algebra $\mathcal{K}_n^+$ (resp. $\mathcal{K}_n$) and $M$ be a  $\gl_n$-module. Then the tensor product
$$F(P, M)=P\otimes_{\C} M$$
becomes a $ \widehat{
W}_n^+$-module (resp. $W_n$-module) with the action
$$(d+ fA)(g\otimes v)=d(g)\otimes v+fg\otimes
A(v)$$ for $d\in  W_n,\;f\in{A}_n, g\in P,\;A\in
\gl_n$ and $v\in M.$ Thus we have the following
$W_n^+$-module (resp. $W_n$-module) structure on $F(P, M)$:
$$y\circ (g\otimes v)=\tau(y)(g\otimes v), \ \forall y\in
W_n,\;g\in V ,\;v\in M.$$
Particularly we have
\begin{equation} \label{Action1}(t^{\a}\partial_{j})\circ (g\otimes v)=((t^{\a}\partial_{j})g)\otimes v+ \sum_{i=1}^n(\ptl_{i}(t^{\a})g)\otimes E_{ij}(v)\end{equation}
for all  $\alpha\in \Z^n, g\in P$ and $v\in M.$

From \begin{equation} \label{Action2}(t_j\partial_{j})\circ (g\otimes v)=(t_j\partial_{j}g)\otimes v+ g\otimes E_{jj}(v)\end{equation}
we can see  that $F(P, M)$ is a weight  module if and only if $M$ is a weight $\gl_n$-module and $P$ is a weight $\mathcal{K}_n^+$-module (resp. $\mathcal{K}_n$-module).

When $M$ is the $1$-dimensional trivial $\gl_n$-module, from Theorems 6 and  12 in \cite{TZ} we know that $F(P, M)$ is an irreducible  $W_n^+$-module (resp. $W_n$-module) if and only if $P$ is a simple $\K_n^+$-module (resp.  $\K_n$-module) that is not isomorphic to the natural $\K_n^+$-module $A_n^+$ (resp. $\K_n$-module $A_n$).

In this paper, we will determine  necessary and sufficient conditions for the module $F(P, M)$ to be irreducible over $W_n^+$ (resp. over $W_n$).

\section{Irreducibility of  the modules $F(P, M)$}

In this section, we will  find   necessary and sufficient conditions for the  module $F(P, M)$ to be irreducible over $W_n^+$ (resp. over $W_n$), as well as determine all submodules of $F(P, M)$ when
it is reducible.

\begin{theorem}\label{thm-2.1} Let $P$ be an irreducible $\K_n^+$-module (resp. $\K_n$-module),  $M$  an irreducible $\gl_n$-module that is not isomorphic to $V(\delta_r, r)$ for any $r\in \{0, 1,\cdots, n\}$. Then $F(P, M)$ is an irreducible module over $W_n^+$ (resp. over $W_n$).
\end{theorem}

\begin{proof} We will prove this result only for $W_n^+$ since the proof is valid also for $W_n$, even simpler.

Let $V$ be a nonzero proper $W_n^+$-submodule of $F(P,M)$.
 Let $\sum_{k =1}^q p_k\otimes w_k$ be a nonzero element in $V$. (Later we will assume  that $p_1, \dots, p_q$ are linear independent.)

\

\noindent {\bf Claim 1}.  For any $u\in \K_n^+$, $1\le i, j, l\le n$, we have
\begin{equation*}\sum_{k =1}^q(u p_k)\otimes (\delta_{li}E_{lj}- E_{li} E_{lj})w_k\in V.\end{equation*}

 For any $m\in\{0, 1,2\}$,  and   $\beta=(b_1,\ldots,b_n)\in \Z_+^n$ with $b_l\ge 2$,   we have
 $$\aligned& (t^{m\epsilon_l}\partial_i)\hskip -3pt \circ \hskip -3pt  (t^{\beta-m\epsilon_l}\partial_j)\hskip -3pt \circ \hskip -3pt  \sum_{k =1}^q (p_k\otimes w_k)\\
& =\sum_{k =1}^q(t^{m\epsilon_l}\partial_i)\hskip -3pt \circ \hskip -3pt \big(t^{\beta-m\epsilon_l}\partial_j p_k\otimes w_k+\sum_{s=1}^n\ptl_s(t^{\beta-m\epsilon_l})p_k\otimes E_{sj}w_k\big)\endaligned$$ $$\aligned
& =\hskip -3pt \sum_{k =1}^q(t^{m\epsilon_l}\partial_i)\hskip -3pt  \circ\hskip -3pt   \big(t^{\beta-m\epsilon_l}\partial_j p_k\otimes w_k\hskip -3pt +\hskip -3pt  \sum_{s=1}^n(\be_s-m\delta_{s,l})(t^{\beta-m\epsilon_l-\epsilon_s}p_k\otimes E_{sj}w_k)\big)\\
& =\sum_{k =1}^q(t^{m\epsilon_l}\partial_it^{\beta-m\epsilon_l}\partial_j p_k\otimes w_k+mt^{\beta-\epsilon_l}\partial_j p_k\otimes E_{li}w_k)\\
&  \  \  \  +\sum_{k =1}^q\sum_{s=1}^n(\be_s-m\delta_{s,l})\big(t^{m\epsilon_l}\partial_it^{\beta-m\epsilon_l-\epsilon_s}p_k\otimes E_{sj}w_k
\\
&  \  \  \  \  \  \ +mt^{\beta-\epsilon_l-\epsilon_s}p_k\otimes E_{li}E_{sj}w_k\big)\endaligned $$
 $$\aligned&   =\hskip -3pt\sum_{k =1}^q\hskip -3pt\Big((\be_i-m\delta_{il})t^{\be-\epsilon_i}\ptl_jp_k\otimes w_k\hskip -3pt+\hskip -3ptt^\be\ptl_i\ptl_jp_k\otimes w_k\hskip -3pt+\hskip -3ptmt^{\beta-\epsilon_l}\partial_j p_k\otimes E_{li}w_k\Big)\\
&\ +\sum_{k =1}^q\sum_{s=1}^n(\be_s-m\delta_{s,l})\Big((\be_i-m\delta_{il}-\delta_{is})t^{\be-\epsilon_s-\epsilon_i}p_k\otimes E_{sj}w_k
\\
&\  \  \ +t^{\be-\epsilon_s}\ptl_ip_k\otimes E_{sj}w_k+mt^{\beta-\epsilon_l-\epsilon_s}p_k\otimes E_{li}E_{sj}w_k\Big)\\
&=  m^2 \sum_{k =1}^q t^{\beta-2\epsilon_l}p_k\otimes (\delta_{li}E_{lj}- E_{li} E_{lj})w_k+m u_1+u_0\in V,\endaligned $$ where $u_1,u_0$ is determined by $\beta,i,j, w$ and independent of $m$. Taking $m=0,1,2$ and considering the coefficient of $m^2$,  we have
that
\begin{equation}\label{T1}T_{\a,i,j}\hskip -3pt \circ \hskip -3pt (\sum_{k =1}^q p_k\otimes w_k)= \sum_{k =1}^q t^{\a}p_k\otimes (\delta_{li}E_{lj}- E_{li} E_{lj})w_k\in V, \end{equation}
for all $\a\in \Z_+^n, l=1,2,\ldots,n$, where $$T_{\a,i,j}=\frac{1}{2}((t^{2\epsilon_l}\partial_i)(t^{\a}\partial_j))-2(t^{\epsilon_l}\partial_i)(t^{\a+\epsilon_l}\partial_j)+(\partial_i)(t^{\a+2\epsilon_l}\partial_j)\in U(W_n^+).$$
Consequently,  from the action of $\ptl_s$ on $M$, we see that
 \begin{equation}\label{T2}\sum_{k =1}^q(\ptl_s t^{\a}p_k)\otimes (\delta_{li}E_{lj}- E_{li} E_{lj})w_k\in V, \ \ \forall \a\in \Z_+^n, s=1,\dots,n.\end{equation}

 From (\ref{T1}),  (\ref{T2}) and the fact that $\K_n^+$ is generated by $t_i, \ptl_s$, with $i, s=1,\dots,n$, we deduce that
 \begin{equation}\label{T3}\sum_{k =1}^q(up_k)\otimes (\delta_{li}E_{lj}- E_{li} E_{lj})w_k\in V, \ \ \forall\  u\in \K_n^+.\end{equation} Then Claim 1 follows.

\

\noindent  {\bf Claim 2}.  If $p_1, \dots, p_q$ are linear independent, then for any  $1\le i, j, l\le n$, and any $k=1,\dots,q$,  we have
  $$ (\delta_{li}E_{lj}- E_{li} E_{lj})w_k=0.$$

 Since $P$ is an irreducible $\K_n^+$-module,  by the density theorem in ring theory,  for any
$p\in P$,  there is a $u_k\in \K_n^+$ such that
$$u_kp_i=\delta_{ik}p, \ i=1,\dots, q.$$
From Claim 1 we see that  $P \otimes (\delta_{li}E_{lj}- E_{li} E_{lj})w_k\subseteq V.$

Let $M_1=\{w\in M| P \otimes w\subseteq V\}$.  For any $ p\otimes w\in V$, from
$$(t_j\partial_i)\hskip -3pt \circ \hskip -3pt (p\otimes w)=(t_j\partial_i p\otimes w+ p\otimes  E_{ji}(w))\in V, \forall \ w\in M_1,$$
 we see that $M_1$ is a $\gl_n$-submodule of $M$, which has to be $0$ or $M$.
Since $V$ is a proper submodule of $F(P,M)$, we must have that
$$M_1=0.$$ Claim 2 follows.

\

From now on we assume that $p_1, \dots, p_q$ are linear independent.

\noindent {\bf Claim 3}.   For any  $1\le i, j, l\le n$,  we have  $(\delta_{li}E_{lj}- E_{li} E_{lj})(M)=0$.

From $$t_s\ptl_m \hskip -3pt \circ \hskip -3pt  (\sum_{k =1}^q p_k\otimes w_k)\in V,$$ we have
$$\sum_{k =1}^q (t_s\ptl_m p_k\otimes w_k+  p_k\otimes E_{sm}w_k)\in V.$$

By Claim 1,  we obtain $$\sum_{k =1}^q (ut_s\ptl_m p_k\otimes (\delta_{li}E_{lj}- E_{li} E_{lj})w_k+  up_k\otimes (\delta_{li}E_{lj}- E_{li} E_{lj})E_{sm}w_k)\in V.$$
By Claim 2 we have
$$\sum_{k =1}^q up_k\otimes (\delta_{li}E_{lj}- E_{li} E_{lj})E_{sm}w_k\in V,$$ for all $u\in \K_n^+$.
Since $p_i$'s are linearly independent,  by taking different $u$ in the above formula, we deduce that  $$P \otimes (\delta_{li}E_{lj}- E_{li} E_{lj})E_{sm}w_k\in V, \ \  k=1,\dots,q.$$
This means that
$$(\delta_{li}E_{lj}- E_{li} E_{lj})E_{sm}w_k\subseteq M_1,\   \forall\  l, i, j, s, m . $$
Now $M_1=0$. So  $$(\delta_{li}E_{lj}- E_{li} E_{lj})E_{sm}w_k=0, \forall\  l, i, j, s, m, k. $$
By repeatedly doing this procedure we deduce that
 $$(\delta_{li}E_{lj}- E_{li} E_{lj})U(\gl_n)w_k=0, \forall\  l, i, j,  k. $$
Since $M$ is an irreducible $\gl_n$-module,   we obtain that $$(\delta_{li}E_{lj}- E_{li} E_{lj})M=0, \forall\  l, i, j. $$
Claim 3 follows.

Consequently $$(E_{ii}-1)E_{ii}M=E_{li}^2M=0,  \forall\ i, l \ \text{with}\   i\neq l.$$

By Lemma 2.3 in \cite{LZ}, we see that   $M$ is finite dimensional, i.e., $M$ is a highest weigh $\gl_n$-module with highest weight $\mu\in \Lambda^+$.
 Suppose that  $v_\mu$ is a highest weight vector of $M$, i.e.,
$E_{ii}v_\mu=\mu_i v_\mu$ for $i=1,\dots, n$. Then $\mu_i (\mu_i -1)=0$, i.e., $\mu_i=0$ or $1$. Since $\mu\in \Lambda^+$, we see that $\mu_i\ge \mu_{i+1}$. Thus there is an $r: 1\le r\le n$ such that $\mu_i=1$ for all $i\le r$, and $\mu_i=0$ for all $i> r$. So  $M$ has highest weight $\delta_r$ with $b=r$, contradicting the assumption on $M$. The theorem follows.
\end{proof}

We know that $V(\delta_0,b)=\C v$  is  the one dimensional $\gl_n$-module such that the identity matrix $I$ acts by the scalar $b$. Note that
$V(\delta_0,n)=V(\delta_n,n)$.
As vector spaces $F(P, V(\delta_0,b))\cong P$. By (\ref{Action1}), the action of $W_n^+$ (resp. $W_n$) on $F(P, V(\delta_0,b))$ is given by
$$ (t^{\a}\partial_{j})\hskip -3pt \circ \hskip -3pt (g)=t^{\a}(\partial_{j}g)+ \frac{b\a_j}{n}t^{\a-\epsilon_j}(g), \ \ \forall\ g\in P.$$  By Theorem \ref{thm-2.1},  the   module  $F(P, V(\delta_0,b)) $  over $W_n^+$   (resp. $W_n$)  is irreducible  if $b\neq 0, n$.

Next we will study the structure of the $W_n^+$-modules (resp. $W_n$-modules) $F(P, V(\delta_k,k))$ for $k=0, 1, 2, \cdots, n$.

\begin{lemma}\label{2.2}Let $P$ be a module over the associative algebra   $\mathcal{K}_n^+$  (resp. $\mathcal{K}_n$). Then for $k=0,1,2,\ldots,n-1$,  we have  $W_n^+$-module (resp. $W_n$-module) homomorphisms
 $$\aligned \pi_k:\,\, &F(P, V(\delta_k,k))\rightarrow  F(P,V(\delta_{k+1},k+1)),\\ &p\otimes w\mapsto \sum_{l=1}^n(\partial_l\cdot p)\otimes  (\epsilon_l\wedge w), \forall p\in P, w\in M.\endaligned$$ Furthermore,
 $\pi_{k+1}\pi_k=0.$\end{lemma}

\begin{proof}
From $$(t^{\beta}\partial_i)\hskip -3pt \circ \hskip -3pt (p\otimes w)=((t^{\beta}\partial_i)\cdot p)\otimes w+\sum_{s=1}^n (\partial_s(t^{\beta})\cdot p)\otimes E_{si}w,$$
we have $$\pi_k((t^{\beta}\partial_i)\hskip -2pt \circ \hskip -2pt (p\otimes w))\hskip -3pt =\hskip -4pt \sum_{l=1}^n\hskip -3pt  (\partial_l t^{\beta}\partial_i) p\otimes (\epsilon_l\wedge w)\hskip -3pt +\hskip -3pt \sum_{l=1}^n\hskip -3pt \sum_{s=1}^n \hskip -3pt \partial_l \partial_s(t^{\beta}) p\otimes \epsilon_l\wedge E_{si}w.$$
We compute
$$\aligned& (t^{\beta}\partial_i)\hskip -3pt \circ \hskip -3pt \pi_k(p\otimes w)=(t^{\beta}\partial_i)\hskip -3pt \circ \hskip -3pt (\sum_{l=1}^n(\partial_l\cdot p)\otimes  (\epsilon_l\wedge w))\\
&=\sum_{l=1}^n( (t^{\beta}\partial_i)\cdot\partial_l\cdot p)\otimes  (\epsilon_l\wedge w) +\sum_{l=1}^n\sum_{s=1}^n \partial_s(t^{\beta})(\partial_l\cdot p)\otimes E_{si}(\epsilon_l\wedge w)\\
&=\sum_{l=1}^n( (\partial_lt^{\beta}\partial_i-\partial_l(t^{\beta})\partial_i)\cdot p)\otimes  (\epsilon_l\wedge w)\\
& \  \  \ +\sum_{l=1}^n\sum_{s=1}^n \partial_s(t^{\beta})(\partial_l\cdot p)\otimes (E_{si}\epsilon_l)\wedge w\\
&\ \  \ \ \ +\sum_{l=1}^n\sum_{s=1}^n \partial_s(t^{\beta})(\partial_l\cdot p)\otimes (\epsilon_l\wedge E_{si} w)\\
&=\sum_{l=1}^n ((\partial_lt^{\beta}\partial_i-\partial_l(t^{\beta})\partial_i)\cdot p)\otimes  (\epsilon_l\wedge w)\\
&\  \ +\sum_{s=1}^n \partial_s(t^{\beta})(\partial_i\cdot p)\otimes (\epsilon_s \wedge w)+\sum_{l=1}^n\sum_{s=1}^n (\partial_l\cdot \partial_s(t^{\beta})\cdot p)\otimes (\epsilon_l\wedge E_{si} w)\\ &\  \ -\sum_{l=1}^n\sum_{s=1}^n (\partial_l(\partial_s(t^{\beta}))\cdot p)\otimes (\epsilon_l\wedge E_{si} w)\\
&=\pi_k((t^{\beta}\partial_i)\hskip -3pt \circ \hskip -3pt (p\otimes w)).\endaligned $$
 Note that in the last equation we have used $\epsilon_l\wedge E_{si} w+\epsilon_s\wedge E_{li} w=0$ for all $s,l=1,2,\ldots,n.$
 It is easy to see that $\pi_{k+1}\pi_k=0.$
\end{proof}

\begin{corollary} Let $P$ be an irreducible module over the associative algebra $\mathcal{K}_n^+$  (resp. $\mathcal{K}_n$-modules).  Then the $W_n^+$-module (resp. $W_n$-module) $F(P, V(\delta_r,r))$ is not irreducible for $r=1,2,\ldots,n-1$.\end{corollary}

\begin{proof} If $\partial_i(P)=0$ for some $i=1,2,\cdots, n$, from
$$0=\partial_i(t_ip)=p+t_i(\partial_i(p)),\forall p\in P$$
we deduce that $P=0$ which contradicts the fact that $P$ is simple. Thus $\partial_i(P)\ne 0$ for  $i=1,2,\cdots, n-1, n$.
From Lemma \ref{2.2}, we see that ${\rm Im} \pi_i\ne0$  for $i=1,2,\cdots, n-1$.
 Then from $\pi_{r+1}\pi_r=0$ for $r=0,\ldots,n-2$, we have ${\rm Im} \pi_r$ is a proper submodule of $F(P,V(\delta_{r+1},r+1))$. Thus the module $F(P, V(\delta_r,r))$ is not simple for $r=1,2,\ldots,n-1$.
\end{proof}

For simplicity, let $L_n(P,r)=\pi_{r-1}(F(P, V(\delta_{r-1},r-1))$ for $r=1,2,\ldots,n$ and set $L_n(P,0)=0$.  By the definition of $\pi$,  we can see that $L_n(P,r)$  is spanned by
\begin{equation}\label{2.4'}\sum_{k=1}^n(\ptl_kp)\otimes \epsilon_k\wedge \epsilon_{i_2}\wedge\cdots\wedge \epsilon_{i_r}=\sum_{k=1}^n(\ptl_k p)\otimes E_{kj}w,  \end{equation} where   $p\in P$, and   $j$ is chosen so that $w=\epsilon_j\wedge  \epsilon_{i_2}\wedge\cdots\wedge \epsilon_{i_r}\ne 0$. Therefore for all $j=0,1,2,\ldots,n; p\in P, w\in V(\delta_r,r)$ we have
\begin{equation}\sum_{k=1}^n(\ptl_k p)\otimes E_{kj}w\in L_n(P,r), \end{equation}
\begin{equation}\label{2.6} t^{\gamma}\partial_j(p\otimes w)+L_n(P,r)=\sum_{s=1}^n t^{\gamma}\partial_s p\otimes (\delta_{j,s}w-E_{sj}w)+L_n(P,r),\end{equation}
for any $\gamma\in\Z^n$.
Let $$\tilde{L}_n(P,r)=\{v\in F(P, V(\delta_r,r))  \ | \ W_n^+ v\subseteq L_n(P,r)\}$$
for   $W_n^+$, and  $$\tilde{L}_n(P,r)=\{v\in F(P, V(\delta_r,r))  \ | \ W_n v\subseteq L_n(P,r)\}$$
for   $W_n$.
 Since $W_n^+=[W_n^+, W_n^+]$ and $W_n=[W_n, W_n]$, we see
 \begin{equation}\label{2.9}\tilde{L}_n(P,r)=\{ v\in F(P, V(\delta_r,r))  \ | \ W_n^+ v\subseteq \tilde{L}_n(P,r)\}\end{equation}
for   $W_n^+$, and  $$\tilde{L}_n(P,r)=\{ v\in F(P, V(\delta_r,r))  \ | \ W_nv\subseteq \tilde{L}_n(P,r)\}$$
for   $W_n$. It is clear that $\tilde{L}_n(P,r)/L_n(P,r)$ is either zero or a trivial module.

\begin{lemma}\label{lemma-2.5}Let $P$ be an irreducible  module over the associative algebra   $\mathcal{K}_n^+$ (resp.  $\mathcal{K}_n$).  Then the module  $F(P,V(\delta_r,r))/\tilde{L}_n(P,r)$ over $\W_n^+$  (resp. $\W_n$) is either 0 or irreducible for any $r=0,1,2,\ldots,n$.\end{lemma}

\begin{proof} Suppose that $F(P,V(\delta_r,r))/\tilde{L}_n(P,r)\ne 0$. Let $V$ be any submodule of $F(P,V(\delta_r,r))$ with $\tilde{L}_n(P,r)\subsetneq V$.
Take $$v=\sum_{k=1}^q p_k\otimes w_k \in V\backslash \tilde{L}_n(P,r).$$ From (\ref{2.9})
there exists some $t^{\beta}\partial_j$ such that $t^{\beta}\partial_j v\not\in \tilde{L}_n(P,r)$. From (\ref{2.6}) we have
  \begin{equation}\label{2.9'}t^{\gamma+\beta}\partial_j\cdot v \equiv \sum_{k =1}^q \sum_{s=1}^n t^{\gamma+\beta}\partial_s p_k\otimes(\delta_{j,s}w_k-  E_{sj}w_k)\,\, {\rm mod}\,\, \tilde{L}_n(P,r).   \end{equation}
Let $v_1=\sum_{k =1}^q \sum_{s=1}^nt^{\beta}\partial_s p_k\otimes(\delta_{j,s}w_k-  E_{sj}w_k)$ which is nonzero. Taking $\gamma=0$ in \eqref{2.9'} we see that  $v_1\not\in \tilde{L}_n(P,r)$. Noting that $t^{\gamma+\beta}\partial_j\cdot v\in V$ and from \eqref{2.9'} we see that
 $$ \sum_{k =1}^q \sum_{s=1}^nxt^{\beta}\partial_s p_k\otimes(\delta_{j,s}w_k-  E_{sj}w_k)\in V, \forall \ x\in \mathcal{K}_n^+ (\text{resp. }\forall x\in \mathcal{K}_n).$$ Now by the Density Theorem, we have $P\otimes w \subseteq V$ for some $0\ne w\in V(\delta_r,r)$. Let $M_1=\{w\in V(\delta_r,r)| P \otimes w\subseteq V\}$. We see that $M_1\ne0$. Again from
$$(t_j\partial_i)\hskip -3pt \circ \hskip -3pt (p \otimes w)=(t_j\partial_i p\otimes w+ p\otimes  E_{ji}(w))\in V, \forall w\in M_1, p\in P,$$
 we see that $M_1$ is a $\gl_n$-submodule of $V(\delta_{r},r)$. Thus $M_1=V(\delta_r,r)$. Consequently, $V=F(P, V(\delta_{r},r))$. Then   $F(P,V(\delta_r,r))/\tilde{L}_n(P,r)$ is an irreducible  $\W_n^+$-module (resp. $\W_n$-module).\end{proof}

\begin{theorem}\label{3.5} Let $P$ be an irreducible module over the associative algebra    $\mathcal{K}_n^+$ (resp. $\mathcal{K}_n$). Then
\begin{itemize}
\item[(1).] $\tilde{L}_n(P,r)={\rm ker} \pi_r$ for all $r=0,1,2,\ldots,n-1;$
\item[(2).]    $F(P, V(\delta_{r},r))/\tilde{L}_n(P,r)$ are irreducible for all $r=0,1,\ldots,n-1$;
\item[(3).]  $L_n(P,r)$   are irreducible for all $r=1,2,\ldots,n$, and further, $$F(P, V(\delta_{r-1},r-1))/\tilde{L}_n(P,r-1)\simeq L_n(P,r);$$
\item[(4).] $F(P,V(\delta_0,0))=P$ is irreducible if $P$ is not isomorphic to the natural module $A_n^+ $ (resp. $A_n$), and $A_n^+/\C t^0$  (resp. $A_n/\C t^0$) is irreducible;
\item[(5).] $F(P,V(\delta_n,n))$ is irreducible if and only if $\sum_{k=1}^n \partial_k P=P$. If $F(P,V(\delta_n,n))$ is not irreducible, then $F(P,V(\delta_n,n))/L_n(P,n)$ is a trivial module.
\end{itemize}
\end{theorem}

\begin{proof}(1). We first prove that $\tilde{L}_n(P,r)\subseteq {\rm ker} \pi_r$.

For any $v\in \tilde{L}_n(P,r)$, we know that $(t^{\beta}\partial_j) v\in L_n(P,r)$. Then $(t^{\beta}\partial_j)(\pi_rv)=\pi_r((t^{\beta}\partial_j)v)=0$. In particular, we have  $\partial_j(\pi_rv)=0,$ for all $j=1,2,\ldots,n$, which implies $\pi_r(v)=0$ if $P\not\cong {{A}_n^+ }$  (resp. $P\not\cong A_n $). Now suppose that $P\cong A_n^+$ (resp. $P\cong A_n $). Then we have
  $$\pi_{r}(v)=1\otimes w\in 1\otimes V(\delta_{r+1},r+1)$$ for some $w\in  V(\delta_{r+1},r+1)$.
 Thus $$\aligned \pi_{r}(v)=1\otimes w=&\frac{1}{r+1}(t_1\partial_1+t_2\partial_2+\ldots+ t_n\partial_n)(1\otimes w)\\
 =&\frac{1}{r+1}(t_1\partial_1+t_2\partial_2+\ldots+ t_n\partial_n)\pi_{r}(v)\\
 =& \ 0.\endaligned $$ So $\tilde{L}_n(P,r)\subseteq {\rm ker} \pi_r$.

Since $\pi_r\ne0$, using Lemma \ref{lemma-2.5} we deduce that $\tilde{L}_n(P,r)={\rm ker} \pi_r,$ for all $r=0,1,2,\ldots,n-1.$ Part (1) follows.

Now we consider the module homomorphism sequence:
$$\aligned 0\to F(P,V(\delta_0,0))&\to F(P,V(\delta_1,1))\to \cdots \\
&\to F(P,V(\delta_{n-1},n-1))\to F(P,V(\delta_n,n))\to 0.\endaligned
$$

Part  (2) follows from Lemma \ref{lemma-2.5} and the fact that $\pi_r\ne0$.

Part  (3) follow from Lemma \ref{lemma-2.5}, the above module homomorphism sequence  and (1).

 (4). This follows from Theorems 6 and  12 in \cite{TZ}.

 (5). From (3) and Lemma \ref{lemma-2.5}, we know that $F(P,V(0,n))$ is simple if and only if $F(P,V(0,n))=L_n(P, n)$. From \eqref{2.4'} we have
 $$L_n(P, n)=(\sum_{i=1}^n \partial_i P)\otimes (\epsilon_1\wedge\cdots \wedge \epsilon_n).$$
 Thus $F(P,V(0,n))$ is simple if and only if $P=\sum_{i=1}^n \partial_i P$.

The other statement in Part (5) is obvious. This completes the proof of Part (5).
 \end{proof}

%
%
%
%
%
%

Now we summarize our established results into the following

\begin{corollary}\label{Cor3.6}Let $P$ be an irreducible $\K_n^+$-module (resp. $\K_n$-module),  $M$  an irreducible module over $\gl_n$.  Then
\begin{itemize}\item[(a).] $F(P, M)$ is an irreducible $W_n^+$-module  (resp. $W_n$-module)
 if  $M\not\cong V(\delta_r, r)$ for any $r\in \{0, 1,\cdots, n\}$.
\item[(b).] $F(P,V(\delta_0,0))=P$ is irreducible if and only if  $P$ is not isomorphic to the natural module $A_n^+ $ (resp. $A_n$).
\item[(c).] $F(P,V(\delta_n,n))$ is irreducible if and only if $\sum_{k=1}^n \partial_k P=P$.
\item[(d).] $F(P,V(\delta_r,r))$ is not irreducible if $r=1, 2, \cdots, n-1$.
\end{itemize}
\end{corollary}

Next, we will give the isomorphism criterion for two irreducible modules $F(P,M)$.

\begin{lemma} Let $P,P'$ be irreducible $\mathcal{K}_n^+$-modules (resp. $\mathcal{K}_n$-modules) and $M,M'$ be irreducible $\gl_n(\C)$-modules.
 Suppose that $M\not\cong V(\delta_r,r)$ for $r=0,1,2,\ldots,n$. Then
$F(P,M)\cong F(P',M')$ if and only if $P\cong P'$ and $M\cong M'$.\end{lemma}

\begin{proof} We will prove this result only for $W_n^+$ since the proof is valid also for $W_n$.

The sufficiency is obvious. Now suppose that
 $$\psi:F(P,M)\rightarrow F(P',M')$$ is an isomorphism of $W_n^+$-modules.
 Let $0\ne p\otimes w\in F(P,M)$.
  Write $$\psi(p\otimes w)=\sum_{i=1}^q p_i'\otimes w_i'$$ with $p_1',\ldots,p_q'$ linearly independent. Similar to (\ref{T3}), we have
\begin{equation} \label{T3'}\psi(xp\otimes (\delta_{li}E_{lj}- E_{li} E_{lj})w)=\sum_{k =1}^q xp_k'\otimes (\delta_{li}E_{lj}- E_{li} E_{lj})w_k',\end{equation}
For all $1\le i,j, l\le n$ and $ x\in \mathcal{K}_n^+$.
 Note that we have assumed that $M\not\cong V(\delta_r,r)$ for $r=0,1,2,\ldots,n$. Then  we may assume that $(\delta_{li}E_{lj}- E_{li} E_{lj})w\ne 0$ for some  $l,i,j$. Since $p_1',\ldots,p_q'$ are linearly independent,  from the Density Theorem, we can find $y\in\K_n^+$ so that $yp_i=\delta_{i1}p_1$, i.e.,   $$\psi(yp\otimes (\delta_{li}E_{lj}- E_{li} E_{lj})w)=yp_1\otimes (\delta_{li}E_{lj}- E_{li} E_{lj})w'_1\ne0.$$
 Replacing $x$ with $xy$ in (\ref{T3'}), then replacing  $p$ with $yp$, $ w$ with $(\delta_{li}E_{lj}- E_{li} E_{lj})w$,
 $p_1$ with $p'$,  and $(\delta_{li}E_{lj}- E_{li} E_{lj})w'_1$ with $w'$ we have  that $p\otimes w\ne0$ and
 $$\psi(xp\otimes w)=xp'\otimes w', \forall x\in\K_n^+.$$
Therefore $\psi_1:P\rightarrow P'$ with $\psi_1(xp)=xp'$ is a well-defined map.  Since $P$ and $P'$ are irreducible   ${\mathcal K}_n^+$-modules, so
 $$\text{Ann}_{\K_n^+}(p)=\text{Ann}_{\K_n^+}(p'), {
\text{ and }}P\simeq \K_n^+/\text{Ann}_{\K_n^+}(p)\simeq P'.$$
Thus  $\psi_1$ is an isomorphism of ${\mathcal K}_n^+$-modules.
We have
\begin{equation}\label{3.1} \psi(p\otimes w)=\psi_1(p)\otimes w',\forall p\in P.\end{equation}

Now from  $t_i\partial_j\psi(p\otimes w)=\psi((t_i\partial_j)(p\otimes w))$ and (\ref{3.1}), we deduce that
\begin{equation} \psi(p\otimes E_{ij}w)=\psi_1(p)\otimes E_{ij}w',\forall  p\in P,  i,j=1,2,\ldots,n.\end{equation}
In this manner we obtain that  \begin{equation} \psi(p\otimes yw)=\psi_1(p)\otimes yw',\forall p\in P, y\in U(\gl_n).\end{equation}
So we have
$\text{Ann}_{U(\gl_n)}(w)=\text{Ann}_{U(\gl_n)}(w')$. Since $W$ and $W'$ are irreducible $U(\gl_n)$-modules  we obtain that $$M\simeq U(\gl_n)/\text{Ann}_{U(\gl_n)}(w)\simeq M'.$$
\end{proof}


\section{Examples}

It is known that the Weyl algebra $\mathcal{K}_n^+$ is  a simple Noetherian algebra of Gelfand Kirillov dimension $2n$.
Bernstein
has shown that $GK(P)\geq n$ for any nonzero finitely generated
$\mathcal{K}_n^+$-module $P$, see \cite{BE}. A module $P$ is called holonomic if $GK(P)=n$. Irreducible holonomic $\mathcal{K}_2^+$-modules are
classified in \cite{BV}.
Classifying all irreducible holonomic $\mathcal{K}_n^+$-modules for $n>2$
 is sill open.

 Block  classified
 irreducible $\mathcal{K}_1^+$-modules, up to irreducible elements in $\mathcal{K}_1^+$,  see \cite{B}.
All irreducible $\mathcal{K}_1^+$-modules are holonomic. Note that the $n$-th Weyl algebra $\mathcal{K}_n^+$ is the tensor
product of rank $1$ Weyl algebras $\C[t_1,\partial_1]$, $\C[t_2,\partial_2], \dots, \C[t_n,\partial_n]$.
By tensoring  irreducible $\C[t_i,\partial_i]$-modules $V_i$,  one obtains an irreducible holonomic
$\mathcal{K}_n^+$-module  $V_1\otimes  \cdots\otimes V_n$.

There are a lot of infinite dimensional irreducible $\gl_n$-modules known. See, for example, \cite{FGR}, \cite{MS}, \cite{GS},   \cite{N}.
%
%
\

\noindent{\bf Example 4.1} It is known (see \cite{FGM}) that any irreducible weight $\mathcal{K}_1^+$-module is isomorphic to one of the following
$$t_1^{\lambda_1}\C[t_1^{\pm 1}],\ \lambda_1\not\in\Z, \ \C[t_1],\ \C[t_1^{\pm 1}]/\C[t_1].$$
Since $\mathcal{K}_n^+$ is the tensor product of the rank $1$ Weyl algebras, every irreducible weight $\mathcal{K}_n^+$-module is isomorphic to $V_1\otimes  \cdots\otimes V_n$, where each $V_i$ is an irreducible
 weight $\C[t_i,\partial_i]$-module. By Theorem 2.1, for an irreducible $\gl_n$-module $M$,  we have that $F(V_1\otimes  \cdots\otimes V_n, M)$ is an irreducible module over
$W_n^+$ if   $M\not\cong V(\delta_r, r)$ for any $r\in \{0, 1,\cdots, n\}$. In this way, we have constructed  a lot of irreducible weight modules  with many
different weight supports.

\

\noindent{\bf Example 4.2} For
$\Lambda_n=(\lambda_1,\lambda_2,\cdots,\lambda_n)\in (\C^*)^n$,
denote by $\Omega(\Lambda_n)=\C[x_1,x_2,\cdots,$
 $x_n]$ the polynomial  algebra over $\C$ in commuting indeterminates $x_1,x_2,\cdots,
 x_n.$ The action of $\mathcal{K}_n^+$-module on $\Omega(\Lambda_n)$ is defined by
$$\aligned \partial_{i} \cdot
f(x_1,\cdots,x_n)=&\ \lambda_i^{-1}(x_i+1)f(x_1,\cdots,x_i+1,\cdots,x_n),\\
t_i\cdot f(x_1,\cdots,x_n)=&\ \lambda_if(x_1,\cdots,x_i-1,\cdots, x_n),\endaligned $$
where $f(x_1,\cdots,x_n)\in \Omega(\Lambda_n),$
$i=1,2,\cdots,n$.  Then $\Omega(\Lambda_n)$ is an irreducible Whittaker module over $\mathcal{K}_n^+$, see Example 3.15 in \cite{BO}.
One can see that $t_i\partial_{i} $ acts freely on  $\Omega(\Lambda_n)$. Note that $$\sum_{i=1}^n\partial_i( \Omega(\Lambda_n))\ne  \Omega(\Lambda_n)$$ and $ \Omega(\Lambda_n)\not\simeq A^+_n$. Thus $F(\Omega(\Lambda_n), M)$ is an irreducible non-weight
module over $W_n^+$ if and only if  $M\not\cong V(\delta_r, r)$ for any $r\in \{1,\cdots, n\}$.

\

\begin{center}
\bf Acknowledgments
\end{center}

\noindent Part of the research presented in this paper was carried out during the visit of the
first author to  Soochow University in April of 2016.  G.L. is partially supported by NSF of China (Grants
11301143, 11471294) and  the grant at Henan University(yqpy20140044). R.L. is partially supported by NSF of China (Grants 11471233, 11371134) and a
Project Funded by the Priority Academic Program Development of Jiangsu Higher
Education Institutions. K.Z. is partially supported by  NSF of China (Grants 11271109, 11471233) and NSERC.


\vspace{4mm}

\noindent   \noindent G.L.: Department of Mathematics, Henan University, Kaifeng 475004, China. Email:
liugenqiang@amss.ac.cn

\vspace{0.2cm} \noindent R.L.: Department of Mathematics, Soochow University, Suzhou 215006, Jiangsu,
P. R. China. Email: rencail@amss.ac.cn

\vspace{0.2cm} \noindent K.Z.: Department of Mathematics, Wilfrid
Laurier University, Waterloo, ON, Canada N2L 3C5,  and College of
Mathematics and Information Science, Hebei Normal (Teachers)
University, Shijiazhuang, Hebei, 050016 P. R. China. Email:
kzhao@wlu.ca

\end{document}